\documentclass[11pt, a4paper]{article}
\usepackage{authblk} 
\usepackage{amsthm}
\usepackage{amsmath}
\usepackage{amssymb}
\usepackage{xcolor}
\usepackage{enumerate}
\usepackage{verbatim}
\usepackage{nicefrac}
\usepackage{hyperref}
\usepackage{booktabs} 

\newtheorem{theorem}{Theorem}

\newtheorem{proposition}[theorem]{Proposition}
\newtheorem{corollary}[theorem]{Corollary}
\newtheorem{construction}{Construction}
\newtheorem*{constructionkk}{Construction KK}
\newtheorem*{constructionad}{Construction AD}

\theoremstyle{remark}%
\newtheorem{example}{Example}%
\newtheorem{remark}{Remark}%

\theoremstyle{definition}%

\newcommand{\Z}{\mathbb Z}

\newcommand{\N}{\mathbb N}
\newcommand{\F}{\mathbb F}
\newcommand{\Fqn}{\F_{q^n}}
\newcommand{\Fq}{\F_q}
\newcommand{\Fqto}[1]{\F_{q^{#1}}}
\newcommand{\FqX}{\Fq[X]}
\newcommand{\cyclgroup}[1]{U_{#1}}
\newcommand{\cyclgen}[1]{\zeta_{#1}}
\newcommand{\of}[1]{(#1)}
\DeclareMathOperator{\ord}{ord}
\DeclareMathOperator{\lcm}{lcm}

\renewcommand{\char}{\text {char}}

\providecommand{\keywords}[1]{\small \noindent \textbf{Keywords:} #1}

\begin{document}

\title{Constructing irreducible polynomials recursively with a reverse composition method}

\author[1]{Anna-Maurin Graner \thanks{anna-maurin.graner@uni-rostock.de}}

\author[1]{Gohar M. Kyureghyan\thanks{gohar.kyureghyan@uni-rostock.de}}

\affil[1]{{Institute of Mathematics}, {University of Rostock}, Germany}

\maketitle

\begin{abstract}
	We suggest a construction of the minimal polynomial \(m_{\beta^k}\) of \(\beta^k\in \Fqn\) over \(\Fq\) from the minimal polynomial \(f= m_\beta\) for all positive integers \(k\) whose prime factors divide \(q-1\). The computations of our construction are carried out in \(\Fq\).  The key observation leading to our construction is that for \(k \mid q-1\) holds
	\[m_{\beta^k}\of{X^k} = \prod_{j=1}^{\frac kt} \cyclgen k^{-jn} f \of{\cyclgen k^j X},\]
	where  
	\(t= \max \{m\mid \gcd(n,k): f\of X = g \of {X^m},  g \in \FqX\}\) and \(\cyclgen{k}\) is a primitive \(k\)-th root of unity in \(\Fq\). The construction allows to construct a large number of irreducible polynomials over \(\Fq\) of the same degree. Since different applications require different properties, this large number allows the selection of the candidates with the desired properties.
\end{abstract}

\keywords{ recursive construction, irreducible polynomial, composition method, multiplicative order, \(k\)-th power, characteristic polynomial.}

\section{Introduction}
\label{intro}
Let \(q\) be a prime power and \(\Fq\) the finite field with \(q\) elements. For \(\beta \in \Fqto n\), we denote by \(m_{\beta}\in \FqX\) the minimal polynomial and by \(\chi_{\beta}   \in \FqX\) the characteristic polynomial of \(\beta\) over \(\Fq\). We call \(\beta\) a \textit{proper} element of \(\Fqn\) if \(\beta\in \Fqn\) and there does not exist a proper subfield \(\Fqto{m}< \Fqn\) such that \(\beta \in \Fqto{m}\). For an irreducible polynomial \(f\in \FqX\)  the smallest positive integer \(e\) such that \(f \mid X^e-1\)  or, equivalently, the multiplicative order of all of its roots, is called the \textit{order} of \(f\) and is denoted by \(e = \ord(f)\). If \(f\) has degree \(n\) and  the order of \(f\) equals \(q^n-1\), we call \(f\) a \textit{primitive} polynomial. Furthermore, for \(k\in \N\) we denote by \(U_k\) the group of the \(k\)-th roots of unity over \(\Fq\), that is, the roots of the polynomial \(X^k-1\in \FqX\). Note that \(U_k\) need not be a subset of \(\Fq\), but \(U_k\subseteq \mathbb E\) for an extension field \(\mathbb E\geq \Fq\). If \(\gcd(q,k)=1\), then  \(\vert \cyclgroup{k} \vert=k\) and throughout this paper we will use the notation \(\cyclgen{k}\) for a generating element of \(\cyclgroup{k}\). For a prime \(p\) and an integer \(m\) we denote by \(\nu_p\of{m}\) the \textit{\(p\)-adic valuation of \(m\)}, that is, \(\nu_p(m) = v\) if  \(m = p^v\cdot r\) with \(\gcd(p,r)=1\). \\

The composition method is widely used to construct irreducible polynomials over finite fields, see for example \cite{Cohen1992,Kyuregyan2004,Kyuregyan2006,Kyureghyan2010,McNay1995,Meyn1995,Panario2020,Ugolini2013}. Originally based on a theorem by Cohen \cite{Cohen1969}, with this method one composes an irreducible polynomial with polynomials or rational functions such that the resulting composition is irreducible itself. The composition usually is of higher degree than the initial polynomial. In order to find polynomials with good cryptographic or arithmetic properties, it is of interest to construct a large number of irreducible polynomials of the same degree from which good candidates can be selected.  In \cite{Kyureghyan2020} Kyureghyan and Kyuregyan introduce a recursive construction of irreducible polynomials which reverses the composition method. Here, an irreducible polynomial \(f\) is extracted from the composition \(f\of{X^2}\), which is obtained from the knowledge of its factorization. This construction yields a large number of polynomials of the same degree as the initial polynomial. During our search for possible generalizations of the recursive construction from \cite{Kyureghyan2020} (in this paper Construction KK), we noticed that the composition \(f\of{X^k}\) was studied by Albert \cite{Albert1956} and Daykin \cite{Daykin1965}. We will use the ideas from \cite{Albert1956} and \cite{Daykin1965} to generalize the results and extend the construction from \cite{Kyureghyan2020}.\\

Next we present results from \cite{Daykin1965} and \cite{Kyureghyan2020}. We use a unified notation and terminology so that the similarities of the approaches become visible. The following result  {\cite[Corollary 3]{Kyureghyan2020}} details all the information needed to formulate Construction KK.

	\begin{theorem}[\cite{Kyureghyan2020}]\label{Kyureghyan2020: Corollary 3}
		Let \(q\) be odd and \(f \in \FqX\), \(f \neq X\), be a monic irreducible polynomial of degree \(n\) and order \(e\). Let \(\beta \in \Fqn\) be a root of \(f\). Then the following statements hold:
		\begin{enumerate}[(i)]
			\item There exists a polynomial \(C \in \FqX\) such that \(C\of{X^2} = f\of X \cdot (-1)^n f\of{-X}\). More precisely, \( C \of X = (-1)^n \sum_{j=0}^n \sum_{u=0}^{2j} (-1)^u c_u c_{2j-u} X^j\), where \(c_0, \ldots, c_n\) are the coefficients of \(f\) and \(c_u = 0\) for \(u>n\).
			\item If \(C\) is irreducible, it is the minimal polynomial of \(\beta^2\) over \(\Fq\) and \(\ord \of C = \frac e {\gcd(e,2)}\).
			\item The polynomial \(C \) is irreducible if and only if there does not exist a polynomial \(D\in \FqX\) such that \(f\of X = D \of {X^2}\).
		\end{enumerate}
	\end{theorem}

Theorem \ref{Kyureghyan2020: Corollary 3} can be proved by elementary means  and leads to the following construction, Construction KK, which is the key step of constructions {\cite[Construction 1]{Kyureghyan2020}} and {\cite[Construction 2]{Kyureghyan2020}}. Note that Theorem \ref{Kyureghyan2020: Corollary 3} (iii) allows to determine whether the polynomial \(C\) is irreducible by a simple examination of the coefficients of the polynomial \(f\). 
\newpage

	\begin{constructionkk}[\cite{Kyureghyan2020}]\label{Kyureghyan 2020: Construction}
		Let \(q\) be odd and \(f \in \FqX\), \(f \neq X\), a monic irreducible polynomial of degree \(n\) such that there does not exist a polynomial \(D\in \FqX\) with \(f \of X = D \of {X^2}\). To construct the monic {irreducible} polynomial \(C\in \FqX\) of degree \(n\) over \(\Fq\), do the following steps:
		\begin{enumerate}[Step 1.]
			\item Compute the product \((-1)^n \cdot f\of X \cdot f\of{-X} = C\of{X^2}.\)
		\item Extract \(C\) from the composition \(C\of{X^2}\).
		\end{enumerate} 
	\end{constructionkk}

A similar transformation with \(X^3\) has been studied in \cite{Albert1956} for primitive polynomials over \(\Fq\). The results from \cite{Albert1956} have been generalized in \cite{Daykin1965}. The next theorem shows that the polynomial \(C\) from Theorem \ref{Kyureghyan2020: Corollary 3} and Construction KK is in fact the characteristic polynomial of \(\beta^2\in \Fqn\) over \(\Fq\). This observation will allow us to develop the generalizations of the results in \cite{Kyureghyan2020}.

	\begin{theorem}[{\cite{Daykin1965}}]\label{Daykin: Theorem 1}
		Let \(f\in \FqX\) be a monic irreducible polynomial of degree \(n\) and \(\beta\in \Fqn\) a root of \(f\). Let \(k \in \N\) and \(k'= \frac k {\gcd(q,k)}\). Then  the characteristic polynomial \(\chi_{\beta^k}\in \FqX\) of \(\beta^k\in \Fqn\) over \(\Fq\) satisfies 
		\[\chi_{\beta^k}\of{X^k} = (-1)^{n (k+1)} \prod_{j=1}^{k} f\of{\cyclgen{k'}^j X}.\]
	\end{theorem}

\begin{remark}
	The polynomials \(f\of{\cyclgen {k'}^j X}\) for \(1 \leq j \leq k\) are not necessarily polynomials over \(\Fq\) and need not be irreducible. Thus, in general, Theorem \ref{Daykin: Theorem 1} does not describe the factorization of \(\chi_{\beta^k}\of{X^k}\) into irreducible factors over \(\Fq\).
\end{remark}

	\begin{theorem}[{\cite{Daykin1965}}] \label{Daykin: Theorem 2}
		Let \(f\in \FqX\) be a monic irreducible polynomial of degree \(n\) and order \(e\) and let  \(\beta\in \Fqn\) be a root of \(f\). Then for \(k \in \N\) the characteristic polynomial \(\chi_{\beta^k}\in \FqX\) of \(\beta^k\in \Fqn\) over \(\Fq\) satisfies 
		\(\chi_{\beta^k} = \left( m_{\beta^k} \right)^{\frac nm}\),
		where the minimal polynomial \(m_{\beta^k}\) of \(\beta^k\) over \(\Fq\) has order \(\frac e {\gcd(e,k)}\) and degree \(m\), which is the least positive integer for which \(\frac e {\gcd(e,k)}\) divides \(q^m-1\).
	\end{theorem}
Note that  Theorem \ref{Kyureghyan2020: Corollary 3} (i) and (ii) follow directly from Theorems \ref{Daykin: Theorem 1} and \ref{Daykin: Theorem 2}.\\

Theorems \ref{Daykin: Theorem 1} and \ref{Daykin: Theorem 2} suggest the following construction of \(m_{\beta^k}\) from \(m_\beta\).

\begin{constructionad}\label{Construction: New AD}
	Let \(f\in \FqX\) be a monic irreducible polynomial of degree \(n\) and order \(e\) and let \(\beta \in \Fqn\) be a root of \(f\). Given a positive integer \(k\leq e\) define \(k'=\frac k {\gcd(q,k)}\). To construct the minimal polynomial \(m_{\beta^k}\) of \(\beta^k\in \Fqn\) over \(\Fq\), do the following steps:
	\begin{enumerate}[Step 1.]
		\item Compute the product \[(-1)^{n (k+1)} \prod_{j=1}^{k} f\of{\cyclgen{k'}^j X} = \chi_{\beta^k}\of{X^k}.\]\label{Step: Product}
		\item Extract \(\chi_{\beta^k}\) from the composition \(\chi_{\beta^k}\of{X^k}\).
		\item Determine \(m\), the least positive integer for which \(\frac e {\gcd(e,k)}\) divides \(q^m-1\). \label{Step: Exponent}
		\item Find the factor \(m_{\beta^k}\) in the product  \(\chi_{\beta^k}= \left(m_{\beta^k}\right)^\frac nm\). \label{Step Find m_betak}
	\end{enumerate}

\end{constructionad}

\begin{remark}
	\begin{enumerate}[(a)]
		\item Note that  \(\cyclgen{k'}\) is an element of \(\Fq\) if and only if \(k'\mid q-1\). Therefore, the computations of step \ref{Step: Product} in  Construction AD are carried out in a pure extension field of \(\Fq\) if \(k'\nmid q-1\).
		\item 	Construction AD can also be applied without the knowledge of the order \(e\) of the polynomial \(f\). In that case we replace Steps \ref{Step: Exponent} and \ref{Step Find m_betak} with factorizing \(\chi_{\beta^k}\), which will be an unknown power of the minimal polynomial \(m_{\beta^k}\) of \(\beta^k\) over \(\Fq\).
		\item On the other hand, if the order \(e\) of \(f\) is known, it is possible to avoid the computation intensive Step \ref{Step Find m_betak} by selecting \(k\) such that \(n=m\). Then the characteristic and the minimal polynomial of \(\beta^k\) over \(\Fq\) are equal.
		\item Construction KK does not depend on the knowledge of the order of the intial polynomial \(f\). If used iteratively, it can even  give information on the order as we will discuss later. 
	\end{enumerate}
\end{remark}

In this paper we suggest a construction of the minimal polynomial \(m_{\beta^k}\) of \(\beta^k\in \Fqn\) over \(\Fq\) from the minimal polynomial \(f= m_\beta\) for all positive integers \(k\) whose prime factors divide \(q-1\) which avoids the computation intensive Step \ref{Step Find m_betak} of Construction AD. Additionally, in this construction computations are carried out in \(\Fq\) and it does not depend on the knowledge of the order of the initial polynomial \(f\). While Construction KK only works for finite fields of odd size, our construction can also be used in finite fields of characteristic \(2\) which is attractive for applications in computer science. The key observation leading to our construction is that for \(k \mid q-1\) holds
\[m_{\beta^k}\of{X^k} = \prod_{j=1}^{\frac kt} \cyclgen k^{-jn} f \of{\cyclgen k^j X},\]
where  \(t= \max \{m\mid \gcd(n,k): f\of X = g \of {X^m} \text{ for a polynomial } g \in \FqX\}\).

\section{Theoretical background for the new construction}\label{Properties}

In Theorem \ref{Daykin: Theorem 2} the order of the monic irreducible polynomial \(f=m_\beta\) is used  to determine the degree of the minimal polynomial \(m_{\beta^k}\) or, equivalently, the power to which the minimal polynomial of \(\beta^k\) is taken in the characteristic polynomial of \(\beta^k\in \Fqn\) over \(\Fq\). In this section we describe how to determine this exponent without the knowledge of the order of \(f\).

\begin{remark}\label{Remark gcd(q,k)>1}
	If \(\gcd(q,k)>1\), the coefficients of \(m_{\beta^k}\) can easily be derived from the coefficients of \(m_{\beta^{k'}}\) where \(k'=\frac k {\gcd(q,k)}\). Indeed, Theorem \ref{Daykin: Theorem 2} implies that \(\ord(m_{\beta^k}) = \frac e {\gcd(e,k)} = \frac e {\gcd(e,k')} = \ord (m_{\beta^{k'}})\) and therefore \(\deg (m_{\beta^{k'}}) = \deg (m_{\beta^k}) = m\). Suppose that  \(m_{\beta^{k'}} = \sum_{i=0}^m a_i X^i\) and set \(g = \sum_{i=0}^m a_i^{\gcd(q,k)} X^i\in \FqX\). Then 
	\begin{align*}
		g\of{\beta^k}&= \sum_{i=0}^m a_i^{\gcd(q,k)} \left(\beta^k\right)^i = \sum_{i=0}^m a_i^{\gcd(q,k)} \left(\beta^{k'}\right)^{i\cdot \gcd(q,k)}\\
		& = \left(m_{\beta^{k'}}\of{\beta^{k'}} \right)^{\gcd(q,k)} = 0.
	\end{align*}
	Thus, \(\beta^k\) is a root of  \(g\) and since \(\deg(g)=m=\deg(m_{\beta^k})\) the polynomial \(g\) is the minimal polynomial of \(\beta^k\) over \(\Fq\). That is, \(m_{\beta^k} = \sum_{i=0}^m a_i^{\gcd(q,k)} X^i\).
\end{remark}

Using Remark \ref{Remark gcd(q,k)>1}, we can restrict our discussion to the case that \(\gcd(q,k)=1\). Nontheless, note that all results hold also for integers \(k\) such that \(\gcd(q,k)>1\). The main advantage of considering only the case \(\gcd(q,k)=1\) is that there always exist exactly \(k\) distinct \(k\)-th roots of unity in an extension field \(\mathbb E \geq \Fq\) of \(\Fq\).

\begin{theorem}\label{Theorem: chi_beta^k(X) reducible}
	Let \(k \in \N\) with \(\gcd(q,k)=1\). Further, let  \(\beta \in \Fqn\) be a proper element of \(\Fqn\) and \(\chi_{\beta^k}\) be the characteristic polynomial of \(\beta^k\in \Fqn\) over \(\Fq\).\\	
	Then \(\chi_{\beta^k}  = \left( m_{\beta^k} \right)^t\) for a positive integer \(t\in \N\) if and only if  every root of the polynomial \(\chi_{\beta^k}\of{X^k}\) has multiplicity \(t\). That is, the roots of \(\chi_{\beta^k}(X)\) and the roots of \(\chi_{\beta^k}\of{X^k}\) have the same multiplicity \(t\).
\end{theorem}

\begin{proof}
	Since \(\chi_{\beta^k} \) is the characteristic polynomial of \(\beta^k\) over \(\Fq\), there exists a positive integer \(t\geq 1\) such that \(\chi_{\beta^k}\of{X^k} = \left(m_{\beta^k}\of{X^k}\right)^t\).
	Furthermore, 
	\(m_{\beta^k}\of{X^{k}} = \prod_{i=0}^{\frac nt -1} \left( X^{k} - \beta^{k\cdot q^{i}}\right)\)
	and for every \(i\) the polynomial \(X^{k}-\left(\beta^{q^{i}}\right)^{k}\) has  \(k\) distinct  roots of the form \(\cyclgen{k}^j\beta^{q^{i}}\) in an extension field of \(\Fq\), where \(1 \leq j \leq k\). Thus, 	
	\(\chi_{\beta^k}\of{X^k} = \prod_{i=0}^{\frac nt -1} \prod_{j=1}^{k} \left( X - \cyclgen{k}^j \beta^{q^{i}}\right)^{t}\). Note that the roots \(\cyclgen{k}^j\beta^{q^{i}}\) of \(\chi_{\beta^k}\of{X^k}\) for \(1 \leq j \leq k\) and \(0 \leq i \leq \frac nt -1\) are distinct.  Indeed, if for \(1 \leq j_1,j_2\leq k\) and \(0 \leq i_1,i_2\leq \frac nt -1\) the two roots 
	\( \cyclgen{k}^{j_1} \beta^{q^{i_1}}\) and \(\cyclgen{k}^{j_2} \beta^{q^{i_2}}\) were equal, we would have 	
	\(\left(\beta^{k}\right)^{q^{i_1}} = \left(\cyclgen{k}^{j_1} \beta^{q^{i_1}}\right)^{k} = \left(\cyclgen{k}^{j_2} \beta^{q^{i_2}}\right)^{k} = \left(\beta^{k}\right)^{q^{i_2}} \)
	and since the elements \(\left(\beta^{k}\right)^{ q^{i}}\) of \(\Fqto{\frac nt}\) are distinct, we have \(i_1 = i_2\) and consequently also \(j_1=j_2\). To complete the proof recall that the roots of irreducible polynomials over finite fields are simple. 
\end{proof}

The roots of the polynomial \(\chi_{\beta^k}\of{X^k}\) lie in an extension field of \(\Fq\). Since we later want to work in \(\Fq\), we state the following immediate consequence of Theorem \ref{Theorem: chi_beta^k(X) reducible}.

\begin{corollary}\label{Corollary: chi_beta^k reducible with irreducible factors}
	Let \(k \in \N\) such that \(\gcd(q,k)=1\). Further, let  \(\beta \in \Fqn\) be a proper element of \(\Fqn\) and \(\chi_{\beta^k}\) be the characteristic polynomial of \(\beta^k\in \Fqn\) over \(\Fq\).\\	
	Then \(\chi_{\beta^k}  = \left( m_{\beta^k} \right)^t\) for a positive integer \(t\) if and only if every irreducible factor of \(\chi_{\beta^k}\of{X^{k}}\) over \(\Fq\) appears with multiplicity \(t\). 
\end{corollary}

Let \(f\in \FqX\) be a monic irreducible polynomial of degree \(n\) and \(\beta\in \Fqn\) be a root of \(f\). By  Theorem \ref{Daykin: Theorem 1}, we have
\begin{equation}\label{Eq: Factorization chi_beta^k over Fq}
	\chi_{\beta^k}\of{X^k} = (-1)^{(k+1)n} \prod_{j=1}^k f\of{\cyclgen k^j X} = \prod_{j=1}^k \cyclgen{k}^{-jn} f\of{\cyclgen{k}^j X}. 
\end{equation}
If \(k\mid q-1\), then \(\cyclgroup{k}\) lies in \(\Fq\) and for \(1\leq j \leq k\) the polynomials \(\cyclgen{k}^{-jn}f\of{\cyclgen k ^j X}\) are monic polynomials of degree \(n\) over \(\Fq\). 
The element \(\cyclgen k^{-j}\beta\) is a root of \(\cyclgen k^{-jn}f\of{\cyclgen k^j X}\) and since \(\beta\) is a proper element of \(\Fqn\), the element \(\cyclgen k^{-j}\beta\) also is a proper element of \(\Fqn\). Consequently, the polynomial \(\cyclgen k^{-jn}f\of{\cyclgen k ^j X}\) is the minimal polynomial of \(\cyclgen k^{-j}\beta\) over \(\Fq\) and (\ref{Eq: Factorization chi_beta^k over Fq}) yields the factorization of \(\chi_{\beta^k}\of{X^k}\) into monic irreducible factors over \(\Fq\). With Corollary \ref{Corollary: chi_beta^k reducible with irreducible factors} we obtain that the exponent of the minimal polynomial of \(\beta^k\)  over \(\Fq\) in the characteristic polynomial \(\chi_{\beta^k}\) is   equal to the multiplicity of every polynomial \(\cyclgen{k}^{-jn}f\of{\cyclgen k^j X}\) in the factorization (\ref{Eq: Factorization chi_beta^k over Fq}). Thus, in the case that \(k \mid q-1\), we need to determine under which conditions the polynomials of the form \(\cyclgen k^{-jn}f\of{\cyclgen{k}^j X}\) are equal.  For this we need the following easy proposition.

\begin{proposition}\label{Fact: f(X^k)}
	Let \(k,m \in \N\) such that \(\gcd(q,k)=1=\gcd(q,m)\) and \(f \in \FqX\). Then the following statements hold:
 	\begin{enumerate}[(a)]
		\item There exists \(g\in \FqX\) such that \(f(X)=g(X^k)\) if and only if \(f(X)=f(\cyclgen{k}X)\). \label{enum: Prop f(X^k)}
		\item If there exist polynomials \(g,h\in \FqX\) such that \(f=g\of{X^k}=h\of{X^m}\), then there exists a polynomial \(u \in \FqX\) such that \(f\of{X} = u\of{X^{\lcm(k,m)}}\).
	\end{enumerate} 
\end{proposition}
\begin{proof} 
	\begin{enumerate}[(a)]
		\item If \(f(X)=g(X^k)\), then \(f\of{\cyclgen{k}X} = g \of{\cyclgen{k}^k X^k} = g \of{X^k} = f(X)\). Vice versa, suppose that \(f(X)=f\of{\cyclgen{k}X}\). Then if \(f\of X = \sum_{i=0}^n a_i X^{i}\), we have \(f\of{\cyclgen{k} X} = \sum_{i=0}^n a_i \cyclgen{k}^{i} X^{i}\). Thus, \(\cyclgen{k}^{i} = 1\) for all \(0 \leq i \leq n\) such that \(a_i \neq 0\). Consequently, \(k = \ord\of{\cyclgen{k}} \mid i\) for all \(0 \leq i \leq n\) such that \(a_i \neq 0\).
		\item We know that \(k\mid i\) and \(m\mid i\) for every \(0\leq i \leq n\) such that \(a_i \neq 0\). Then also \(\lcm(k,m)\mid i\) for every such \(i\).
	\end{enumerate}
	
\end{proof}

The following theorem states that it can be seen directly from the non-zero coefficients of the polynomial \(f\), which polynomials of the form \(\cyclgen k^{-jn}f\of{\cyclgen k^j X}\) are equal. 

\begin{theorem} \label{Theorem: k div q-1 and f cyclgen j are equal iff}
	Let \(k \in \N\) such that \(\gcd(k,q)=1\) and let \(f \in \FqX\) be a polynomial of degree \(n\) {such that \(f(0)\neq 0\)}. Set 	\(t= \max \{ m \mid \gcd(n,k): f\of X = g\of{X^m} \text{ for a polynomial } g  \in \FqX\}\). 
	Then for  \(0 \leq j, j'\leq k-1\) the two polynomials \(\cyclgen{k}^{-jn} f \of{\cyclgen{k}^{j}X}\) and \(\cyclgen{k}^{-j'n} f\of{\cyclgen{k}^{j'}X}\) are equal if and only if \(j\equiv j' \mod \frac {k}t\).
\end{theorem}

\begin{proof} ``\(\Leftarrow\)'': Note that since \(t \mid k\) the element \(\cyclgen{k}^{\frac kt} = \cyclgen{t}\) generates the subgroup \(\cyclgroup{t}\) of the \(t\)-th roots of unity of \(\Fq^{\ast}\). If \(j \equiv j'\mod \frac kt\), then \(j-j' = v\cdot \frac kt \) for an integer \(v\) and we have 
	\begin{align*}
		\cyclgen{k}^{-jn} f\of{\cyclgen{k}^{j}X} &= \cyclgen{k}^{-(j-j')n} \cyclgen k^{-j'n} f\of{\cyclgen{k}^{(j-j')}\cyclgen{k}^{j'}X}		&&= \cyclgen k^{-\frac kt \cdot v \cdot n} \cyclgen k ^{-j'n} f\of{\cyclgen{k}^{\frac kt \cdot v} \cyclgen{k}^{j'}X}\\
		&= \cyclgen k^{-k\cdot v \cdot \frac nt} \cyclgen k^{-j'n} f\of{\cyclgen t^{v} \cyclgen k^{j'}X}
		&&=\cyclgen{k}^{-j'n} f\of{\cyclgen t^{v} \cyclgen k^{j'}X}.
	\end{align*}
	From the definition of \(t\) and Proposition \ref{Fact: f(X^k)} follows that \(f\of{X}=f\of{\cyclgen{t}X}\) and therefore also \(f\of X = f\of{\cyclgen t^v X}\). Thus, \(\cyclgen{k}^{-j'n} f\of{\cyclgen t^{v} \cyclgen k^{j'}X} = \cyclgen{k}^{-j'n} f\of{\cyclgen k^{j'}X}\).\\
	
	``\(\Rightarrow\)'': Suppose that \(\cyclgen k^{-jn} f\of{\cyclgen k ^{j}X} = \cyclgen k^{-j'n} f \of{\cyclgen k^{j'}X}\). Then also 
	
	\begin{equation}\label{eq_=f(X)}
		\begin{aligned}
			\cyclgen{k}& ^{-(j-j')n} f \of{\cyclgen k^{j-j'} X}= \cyclgen k^{j'n}\cdot  \cyclgen k^{-jn} f \of{\cyclgen k^{j} \left(\cyclgen{k}^{-j'} X\right) }\\
			& = \cyclgen k^{j'n}\cdot  \cyclgen k^{-j'n} f \of{\cyclgen k^{j'} \left(\cyclgen{k}^{-j'} X\right) } =  f \of{X}
		\end{aligned}
	\end{equation} Let \(f = \sum_{i=0}^n a_i X^{i}\in \FqX\). Then we have \(\cyclgen{k} ^{-(j-j')n} f \of{\cyclgen k^{j-j'} X} =\)\\
	\(  \sum_{i=0}^n a_i \cyclgen{k}^{-(j-j')(n-i)} X^{i}\). For this polynomial to be equal to \(f\of X\), we need \(k \mid (j-j') (n-i)\) for all \(a_i \neq 0\). Note that \(a_0 = f(0)\neq 0\). Consequently, \(k \mid (j-j')\cdot  n\). Let \(d:=\gcd(n,k)\), then \(\frac kd \mid (j-j')\) and there exists \(v\in \N\) such that \(j-j' = v\cdot \frac kd\). Furthermore, the element \(\cyclgen{k}^{\frac kd}=\cyclgen {d}\) generates the subgroup \(\cyclgroup{d}\) of the \(d\)-th roots of unity of \(\Fq\) and we obtain
	\begin{equation}\label{eq_zeta_d}
		\cyclgen{k} ^{-(j-j')n} f \of{\cyclgen k^{(j-j')} X} = \cyclgen k^{-v \cdot \frac kd \cdot d \cdot \frac nd} f \of{\cyclgen k^{v\cdot \frac kd} X} =  f \of{\cyclgen {d}^{v } X}.
	\end{equation}
	If \(l = \frac d {\gcd(d,v)}\), the element \(\cyclgen{d}^{v}= \cyclgen{l}\) generates the set \(\cyclgroup{l}\) of the \(l\)-th roots of unity over \(\Fq\). Equations (\ref{eq_=f(X)}) and (\ref{eq_zeta_d}) yield that \(f(X)=f(\cyclgen{l} X)\). Note that \(\gcd(d,q)=1\) and with Proposition \ref{Fact: f(X^k)} we obtain that 
	\(M:=\{m\mid d: f(X)= g(X^m), g \in \FqX\}\) is equal to the set \(\{m \mid d: f(X)=f(\cyclgen{m}X)\}\)
	and consequently, \(l\in M\). Let \(t:=\max M\). We will prove that \(M\) is in fact the set of all divisors of \(t\). Note that if \(f(X)=f\of{\cyclgen t X}\), also \(f(X)=f\of{\cyclgen t^{i} X}\) for all \(1 \leq i \leq t\) and any divisor \(m\) of \(t\) satisfies that \(\cyclgen m = \cyclgen t^{\frac tm}\). Thus, all divisors of \(t\) are elements of \(M\). Suppose that there exists an element \(m\in M\) such that \(m\) does not divide \(t\). Then for all \(0 \leq i \leq n\) such that \(a_i\neq 0\), we have \(m\mid i\) and \(t\mid i\). Consequently, \(\lcm(m,t) = t \cdot \frac m {\gcd(m.t)} \mid i\) and since both \(m\) and \(t\) divide \(d\), we obtain \(\lcm(t,m)\in M\). But \(\lcm(t,m)>t\), because \(m\nmid t\). This is a contradiction to the choice of \(t\) and \(M\) is in fact the set of all divisors of \(t\). Consequently, the fact \(l \in M\) is equivalent to \(l \mid t\). Recall that \(l = \frac d {\gcd(d,v)}\) and therefore \( \frac d {\gcd(d , v)}  \mid t\) which is equivalent to \(\frac dt \mid \gcd(d,v)\) and this again is equivalent to \(\frac dt \mid v\). Thus, there exists an integer \(w\) such that \(v = \frac dt \cdot w\). Recall that \(v= \frac {j-j'} {\frac kd}\) and we have \( j-j' = \frac kt \cdot w\). Consequently, \(j\equiv j' \mod \frac kt\).
\end{proof}

As a consequence for \(k\mid q-1\) we have the following result.

\begin{corollary}\label{Corollary: Computation of m_beta^k}
	Let \(k \in \N\) such that \(k \mid q-1\) and let \(f\in \FqX\), \(f\neq X\), be a monic irreducible polynomial of degree \(n\). Further, let \(\beta\in \Fqn\) be a root of \(f\) and \(m_{\beta^k}\in \FqX\) be the minimal polynomial of \(\beta^k\in \Fqn\) over \(\Fq\). Set  \(t= \max \{m\mid \gcd(n,k): f\of X = g \of {X^m} \text{ for a polynomial } g \in \FqX\}\). Then 
	\[m_{\beta^k}\of{X^k} = \prod_{j=1}^{\frac kt} \cyclgen k^{-jn} f \of{\cyclgen k^j X}.\] 
\end{corollary}

\begin{proof}
	Using Theorem \ref{Theorem: k div q-1 and f cyclgen j are equal iff}  we can rewrite  equation (\ref{Eq: Factorization chi_beta^k over Fq})   and  obtain that the characteristic polynomial of \(\beta^k\in \Fqn\) over \(\Fq\) satisfies 
	\[\chi_{\beta^k}\of{X^k} = \prod_{j=1}^k \cyclgen k^{-jn} f \of{\cyclgen k^j X} = \left( \prod_{j=1}^{\frac kt} \cyclgen k^{-jn} f\of{\cyclgen{k^j} X}\right)^t\]
	and that the polynomials \(\cyclgen k^{-jn} f \of{\cyclgen k^j X}\) for \(1 \leq j \leq \frac kt\) are distinct. Then  Corollary \ref{Corollary: chi_beta^k reducible with irreducible factors} completes the proof.
\end{proof}

Recall that Construction AD constructs the polynomial \(\chi_{\beta^k}\of{X^k}\) with the formula from Theorem \ref{Daykin: Theorem 1} and then extracts the irreducible factor of the polynomial \(\chi_{\beta^k}\) over \(\Fq\) in order to obtain the minimal polynomial \(m_{\beta^k}\) of \(\beta^k\). Using Corollary \ref{Corollary: Computation of m_beta^k}, in our construction we  directly compute the polynomial \(m_{\beta^k}\of{X^k}\) from which the minimal polynomial \(m_{\beta^k}\) can then easily be extracted.

\begin{remark}\label{Remark: m_beta^k for k prime}
	Note that if \(k\mid q-1\) and \(k\) is prime, then \(t>1\) if and only if \(t=k\). Thus, if \(f\of X = g\of{X^{k}}\) for a polynomial \(g \in \FqX\), then the minimal polynomial of \(\beta^k\) over \(\Fq\) satisfies \(m_{\beta^k}\of X = g \of X\). Otherwise, we obtain \(m_{\beta^k}\) by extracting it from the composition 
	\(m_{\beta^k}\of{X^k} = \prod_{j=1}^k \cyclgen k^{-jn} f\of{\cyclgen{k}^j X } = (-1)^{n(k+1)} \prod_{j=1}^k f \of{\cyclgen k^j X}.\)
\end{remark}

\section{The new recursive construction of $m_{\beta^k}$ from $m_\beta$ }

Observe that for \(k, k_1, k_2\in \N\) such that \(k=k_1\cdot k_2\) and a proper element \(\beta\) of \(\Fqn\), we have \(\beta^k = \left( \beta^{k_1}\right)^{k_2}\) and consequently \(m_{\beta^k}\of {X^{k_2}} = m_{\left(\beta^{k_1}\right)^{k_2}} \of{X^{k_2}}\).
Thus, instead of using the direct computation of \(m_{\beta^k}\) from \(m_\beta\), we can apply Corollary \ref{Corollary: Computation of m_beta^k} recursively. Meaning that we first compute the minimal polynomial of \(\beta^{k_1}\) and then with this polynomial compute \(m_{\left(\beta^{k_1}\right)^{k_2}}\of{X^{k_2}}\) from which \(m_{\beta^k} = m_{\left(\beta^{k_1}\right)^{k_2}}\) can easily be extracted. Using the unique prime factorization of an integer \(k\), we can apply Remark \ref{Remark: m_beta^k for k prime} to suggest a construction for all \(k\in \N\) whose prime factors divide \(q-1\).

\begin{construction}\label{New Construction}
	Let \(k\in \N\) such that \(k = k_1 \cdots k_m\) where \(k_1,\ldots, k_m\) are prime factors of \(q-1\) (which are not necessarily distinct). Further, let \(f\in \FqX\) be a monic irreducible polynomial of degree \(n\). Set \(f_0:=f\). 	For \(1\leq i \leq m\) compute the monic irreducible polynomial \(f_i\) in the following way:\\
	If there exists a polynomial \(g\in \FqX\) such that \(f_{i-1}\of X = g \of{X^{k_i}}\), then \(f_i = g\). Otherwise, compute \[ (-1)^{\deg\of{f_{i-1}} \cdot (k_i+1)} \prod_{j=1}^{k_i} f_{i-1}\of{\cyclgen{k_i}^j X} = f_i\of{X^{k_i}}\] and extract \(f_i\) from the composition. Then \(f_m\) is the minimal polynomial of \(\beta^k\in \Fqn\) over \(\Fq\), where \(\beta\in \Fqn\) is a root of \(f\).
\end{construction}

The main differences between Construction \ref{New Construction} and Construction AD are that all computations of Construction \ref{New Construction} are carried out in \(\Fq\) and the construction relies solely on the examination of the non-zero coefficients of the polynomials \(f_i\) and not on the order of the initial polynomial \(f\). Furthermore, while in Construction AD the  minimal polynomial \(m_{\beta^k}\) needs to be extracted from the characteristic polynomial, it is computed directly in Construction \ref{New Construction}. \\

\begin{remark}
	\begin{enumerate}[(a)]
		\item  All polynomials obtained with Construction \ref{New Construction} are of the same degree \(n\) as the initial polynomial \(f\), if we select integers \(k\) such that \(\gcd(n,k)=1\) or such that  the order \(\frac e {\gcd(e,k)}\) of the minimal polynomial of \(\beta^k\) does not divide \(q^{\frac nt}-1\) for any divisor \(t\) of \(n\), whose prime factors divide \(\gcd(n,k)\).
		\item If there exists a polynomial \(g\in \FqX\) such that \(f=g\of{X^{t}}\) for a prime divisor \(t\) of \(k\), then the minimal polynomial of \(\beta^k\) will be of lower degree. Observe that in this case the polynomial \(f\of{X+a}\) for any element \(a \in \Fq\backslash\{0\}\) will not be a composition with \(X^{m}\) for any positive integer \(m>1\) and could be used instead of \(f\). This fact was proved in \cite{Kyureghyan2020} for \(t=2\). For the convenience of the reader, we include the generalized proof here.
		
		\begin{proof}
			If \(f\of X = \sum_{i=0}^n b_i X^{i}= g\of{X^t}\) for \(t>1\), then \(b_{n-1}=0\) and since \(f\) is monic, we have \(b_n = 1\). Furthermore, 
			\begin{align*}
				f\of{X+a} &=  (X+a)^n + \underbrace{\sum_{i=0}^{n-2} b_i\ (X+a)^{i}}_{\deg(\ldots)< n-1}\\
				&= \sum_{j=0}^n \binom nj a^j X^{n-j} + \sum_{i=0}^{n-2} b_i (X+a)^{i} \\
				&= X^n + na X^{n-1} + \sum_{j=2}^n \binom nj a^j X^{n-j} + \sum_{i=0}^{n-2} b_i (X+a)^{i}.
			\end{align*}
			Since \(\gcd(n,q)=1\), \(\char\of{\Fq}\) does not divide \(n\) from which follows that \(na \neq 0\) and there cannot exist any positive integer \(m>1\) such that \(f\of{X+a}=h\of{X^m}\) for a polynomial \(h \in \FqX\).
		\end{proof}
		
	\end{enumerate}  
\end{remark}

	In \cite{Albert1956} Albert defines a ``cubing transformation", which is  an iterated application of Construction AD for \(k=3\). He notices that if the order \(e\) of the initial polynomial  and \(3\) are coprime, its behaviour is ``periodic''. That is, after a certain amount of iterations it will yield the initial polynomial again. In \cite{Kyureghyan2020} a similar construction for \(k=2\), the repeated application of Construction KK, is presented, which does not need the knowledge of the order \(e\) of the initial polynomial but can even be used to gain information on \(e\). Our results allow to generalize the construction from \cite{Kyureghyan2020} for primes \(k\) satisfying \(k\mid q-1\) by applying Construction \ref{New Construction} iteratively.  
	
	\begin{construction}\label{Construction: k-adic valuation of order}
		Let \(k\) be a prime factor of \(q-1\) and \(f\in \FqX\) a monic irreducible polynomial of degree \(n\). Further let \( w =\nu_k\of{q^n-1} \) be the \(k\)-adic valuation of \(q^n-1\). Set \(f_0:= f\). For \(i\geq 1\) compute the monic irreducible polynomial \(f_i\) in the following way:\\
		If there exists a polynomial \(g\in \FqX\) such that \(f_{i-1}\of{X} = g\of{X^k}\), then \(f_i =g\). Otherwise, compute
		\[(-1)^{\deg(f_{i-1})\cdot (k+1)} \prod_{j=1}^k f_{i-1}\of{\cyclgen k^j X}= f_i\of{X^k}\]
		and extract \(f_i\) from the composition. If \(f_i = f_l\) for an integer \(l\) such that \(0 \leq l \leq w\) and \(l<i\), then stop. 
	\end{construction}

With the notation from Construction \ref{Construction: k-adic valuation of order}, suppose that the construction terminates for the polynomial \(f_{l+s}\) which is equal to \(f_l\), for integers \(s\geq 1\) and  \(0 \leq l \leq \nu_k(q^n-1)\). Then we call the sequence 
\[(f_0, f_1, \ldots, f_{l-1})\]
the \textit{tail} of the construction and the sequence 
\[(f_l, \ldots, f_{l+s-1})\]
the \textit{orbit}. Note that the construction would yield the polynomials of the orbit repeatedly if we continued to iterate through the integers \(i\geq l+s\). Observe that the length of the tail is \(l\) and the length of the orbit \(s\).

\begin{corollary}\label{Corollary: Order of f after Construction 2}
With the notation from Construction \ref{Construction: k-adic valuation of order}, we suppose that Construction \ref{Construction: k-adic valuation of order} terminated after a tail of length \(l\) and an orbit of length \(s\).\\
	
	Then \(\ord\of{f} = k^l\cdot r\) and \(r\) must satisfy
	\begin{enumerate}[(I)]
		\item \(\gcd(k,r)=1\), \label{enum: nu_k of Corollary: ord(f) after constr2}
		\item \(s = \frac {\ord_r(k)} d\) for a divisor \(d\) of \(\deg\of{f_l}\), \label{enum: s of Corollary: ord(f) after constr2}
		\item Furthermore, for an integer \(0 \leq j \leq \deg(f_l)-1\), \(d\) must satisfy \(\ord_r\of{q^j}=d\)  and \(k^{s} \equiv q^j \mod r\). \label{enum: d of Corollary: ord(f) after constr2}
	\end{enumerate}
\end{corollary}
	
	\begin{proof}
		Let \(\beta\in \Fqn\) be a root of \(f\), that is, \(f=m_\beta\) is the minimal polynomial of \(\beta\) over \(\Fq\). Then with Construction \ref{New Construction} we know that \(f_i = m_{\beta^{k^i}}\) for every \(i\geq 0\). Further, let \(\ord(f)=e\) and \(e=k^v\cdot r\) with \(\gcd(k,r)=1\). Then with Theorem \ref{Daykin: Theorem 2} the minimal polynomial of \(\beta^{k^i}\), that is, the polynomial \(f_i\), has order
		\begin{equation}\label{Eq: Ord(f_i)}
			\ord(f_i)= \begin{cases}
				\frac e {k^i}= k^{v-i}\cdot r & \text{for } 0 \leq i \leq v, \\
				r  & \text{for } i \geq v. 
			\end{cases}
		\end{equation}
		Since the order of the polynomials \((f_0, f_1, \ldots, f_{v-1})\) strictly decreases, these polynomials cannot appear twice in the sequence \((f_i)_{i\geq 0}\). Note that \(v\leq w = \nu_k(q^n-1)\). Thus, the polynomial \(f_v\), which is the first polynomial of order \(r\) of the sequence \((f_i)_{i\geq 0}\), is an element of the sequence \((f_0, f_1, \ldots, f_w)\). 
		
		 We need to examine \(\Z_r^*\), the multiplicative group modulo \(r\), to see that \(f_v\)  is the first polynomial to appear twice in the sequence \((f_i)_{i\geq 0}\) and therefore \(v=l\).	The subgroup \(\langle k \rangle \) of \(\Z_r^*\) generated by \(k\) has order \(\ord_r(k)\), which is the multiplicative order of \(k\) modulo \(r\). This implies that \(\beta^{k^{v+\ord_r(k)}} =\beta^{k^v}\) and obviously the minimal polynomials of \(\beta^{k^v}\) and \(\beta^{k^{v+\ord_r(k)}}\) over \(\Fq\)  are equal. Thus, \(f_v = f_{v+\ord_r(k)}\) and we have shown that \(f_v\) does appear again in the sequence.
		
		However, the length \(s\) of the orbit is not always equal to \(\ord_r(k)\). The polynomials \(f_{i_1}\) and \(f_{i_2}\) are equal if and only if \(\beta^{k^{i_1}}\) and \(\beta^{k^{i_2}}\) are \(\Fq\)-conjugates. Thus, it is possible that there exists a positive integer \(u\) smaller  than \(\ord_r(k)\) such that \(\beta^{k^{v+u}}\) is an \(\Fq\)-conjugate of \(\beta^{k^v}\) and the minimal polynomial \(f_{v+u}= m_{\beta^{k^{v+u}} }\) also is equal to \(f_v = m_{\beta^{k^v}}\). To account for this, we choose \(u \in \N\) to be the smallest positive integer that satisfies
		\begin{equation}\label{Eq: select u }
			\langle k \rangle \cap \langle q\rangle = \langle k^u\rangle = \langle q^j \rangle \leq  \Z_r^{\ast} \quad \text{ for an integer } 0 \leq j \leq \deg\of{f_v}-1.
		\end{equation} 
		Note that since \(\langle k^{\ord_r(k)}\rangle = \langle q^0\rangle\) such an integer \(u\) exists and satisfies \(u\leq \ord_r(k)\). Then \(\beta^{k^{v+u}} = \left(\beta^{k^{v}}\right)^{q^j}\) and \(f_{v+u} = f_v\). Moreover, the minimal polynomials of \(\beta^{k^{v+i}} \) for \(0 \leq i \leq u-1\) are distinct because we selected \(u\) to be the smallest positive integer to satisfy (\ref{Eq: select u }). Consequently, \(v = l\), which shows that (\ref{enum: nu_k of Corollary: ord(f) after constr2}) holds, and the length \(s\) of the orbit equals \(u\). 
		
		Set \(d:= \vert \langle q^j \rangle \vert = \vert \langle k^s\rangle \vert\) which is a divisor of \(\deg(f_l)\), since \(\langle q^j\rangle \leq \langle q\rangle\) and \(\vert \langle q \rangle \vert = \deg(f_l)\). Then because of \(\langle k^s\rangle\) being a subgroup of \(\langle k \rangle\), we have \(s = \frac {\vert \langle k \rangle\vert}{\vert \langle k^s\rangle \vert}= \frac {\ord_r(k)} d\) which shows that (\ref{enum: s of Corollary: ord(f) after constr2}) holds. (\ref{enum: d of Corollary: ord(f) after constr2}) follows directly from equation (\ref{Eq: Ord(f_i)}) and our definition of \(d\).
	\end{proof}

	 Note that with equation (\ref{Eq: Ord(f_i)}) in the proof of Corollary \ref{Corollary: Order of f after Construction 2}  the polynomials \(f_i\) for \(0 \leq i \leq l-1\) of the tail of Construction \ref{Construction: k-adic valuation of order}  have order \(k^{l-i}\cdot r\) and all polynomials of the orbit have order \(r\).\\
	 
	 If \(p_1,\ldots, p_m\) are the distinct prime factors of \(q-1\), and \(\ord \of{f} = e = p_1^{v_1} \cdots p_m^{v_m}\cdot r\) with \(\gcd(q,r)=1\) and  \(v_1\geq0, \ldots, v_m\geq 0\). Then Construction \ref{Construction: k-adic valuation of order} allows us to determine the \(p_i\)-adic valuations \(v_1,\ldots,v_m\) of the order of \(f\). Additionally, Corollary \ref{Corollary: Order of f after Construction 2} (\ref{enum: s of Corollary: ord(f) after constr2}) and (\ref{enum: d of Corollary: ord(f) after constr2}) give further conditions on the factor \(r\). In most of our computations the conditions on the factor \(r\) were so restrictive that Construction \ref{Construction: k-adic valuation of order} yielded the exact order \(e\) of \(f\).\\

\begin{remark}
In the original version of \cite{Kyureghyan2020}, the number of distinct polynomials produced by {\cite[Construction 1]{Kyureghyan2020}},  is given as \(\ord_r(2)\) where \(\ord(f) = 2^v  r\) with \(v\geq 0\) and \(r \geq 1\) odd. As we can see from Corollary \ref{Corollary: Order of f after Construction 2}, this number is false, since the authors did not take into consideration that the construction could also yield the minimal polynomials of 
\(\Fq\)-conjugates over \(\Fq\). Similarly, in {\cite[Remark 1]{Kyureghyan2020}} the information about the order of the initial polynomial \(C_0\of X\) obtained by the construction should be changed to: \(2^lt\) where \(t\) is an odd divisor of \(q^n-1\) and \(k-l = \frac {\ord_t(2)} d\) for a divisor \(d\) of \(n\). 	
\end{remark}

\section{Implementation of the construction}

In this section we discuss which polynomials can be obtained from a given initial polynomial \(f\)  with Construction \ref{New Construction}  and how to select the integers \(k\) for which we apply the construction. All discussions in this section are about this fixed polynomial \(f\). Suppose that \(f\) is of degree \(n\), has order \(e\) and \(\beta\in \Fqn\) is a root of \(f\). Then \(\beta\) has multiplicative order \(e\) and  the subgroup \(\langle \beta \rangle = \{\beta^k: 0 \leq k \leq e-1\}\) of \(\Fqn^\ast\) contains all elements of \(\Fqn\) with multiplicative order dividing \(e\). Consequently, the set of all polynomials of the form \(m_{\beta^k}\) for \(k\geq 0\) is in fact
\(\{m_{\beta^k} : 0 \leq k \leq e-1\}\)
and contains all monic irreducible polynomials over \(\Fq\) whose order divides \(e\).\\

Let \(p_1, \ldots, p_m\) be the distinct prime factors of \(q-1\). Then we can apply Construction \ref{New Construction} for any integer \(k\) that is an element of the set 
\[\mathcal A := \{ p_1^{i_1}\cdots p_m^{i_m}: i_1, \ldots, i_m \geq 0\}.\]
Since the element \(\beta\) has multiplicative order \(e\), Construction \ref{New Construction} yields the minimal polynomial of \(\beta^{k\pmod e}\) over \(\Fq\). Thus, the set of polynomials that we can construct with the integers in \(\mathcal A\) is 
\[\mathcal M := \{m_{\beta^{k\pmod e}}: k = p_1^{i_1}\ldots p_m^{i_m}, i_1,\ldots, i_m\geq 0 \}.\]
However, we would like to emphasize that the construction should not be restricted to the elements of \(\mathcal A\) which are smaller than \(e\), here denoted by \(\mathcal A_{< e}\). An integer \(k\in \mathcal A\), \(k\geq e\), can yield a polynomial that cannot be constructed by choosing all elements of \(\mathcal A_{<e}\). This is the case if its representative \(k \pmod e\) in \(\Z_e\) is not an element of \(\mathcal A\) as can be seen from the following example:

\begin{example}
	Let \(\F_8 = \F(a)\) where \(a\) is a root of the monic irreducible polynomial \(X^3+X+1\in \F_2[X]\). We consider the primitive monic irreducible polynomial  \(f = X^5 + aX^4 + X^3 + aX^2 + (a^2 + a)X + a^2\in \F_8[X]\)  of order \(e=32.767 = 7 \cdot 31 \cdot 151\). Since \(8-1=7\), we can apply Construction \ref{New Construction} for all elements of \(\mathcal A = \{7^i:  i\geq 0\}\). Note that we can use the notation of Construction \ref{Construction: k-adic valuation of order} and say that the construction yields a tail of length \(1\) and an orbit of length \(150\). By this we mean that the polynomials \(m_{\beta^{7}}\) and \(m_{\beta^{7^{151}}}\) are equal, where \(\beta\) is a root of \(f\).
	
	The smallest positive integer \(i\) such that \(7^i\) is greater than or equal to \(e\) is \(6\). In fact, \(7^6 \pmod e  = 117.649 \pmod e = 19.348 = 2^2 \cdot 7 \cdot 691 \notin \mathcal A\). Thus, if we had restricted ourselves to \(\mathcal A_{<e}\), we would only have found 5 of the \(151\) possible polynomials.
\end{example}

	The number of polynomials that we can construct with Construction \ref{New Construction}, which is the size of \(\mathcal M\), obviously depends on the size of \(\mathcal A\) considered in \(\Z_e\):
	\[\mathcal A_{\mod e} = \{p_1^{i_1}\cdots p_m^{i_m} \pmod e: i_1,\ldots, i_m \geq 0\}.\]
	Note that in general \(\vert \mathcal M\vert\) is smaller than the size of \(\mathcal A_{\mod e}\), because  in \(\mathcal A_{\mod e}\) exponents can belong to \(\Fq\)-conjugates which then yield the same polynomial multiple times.  \\
	
	 We believe that it is not possible to give a closed formula for \(\vert \mathcal M\vert \) in general since  computing \(\vert \mathcal A_{\mod e}\vert\) is difficult. Indeed, it is related to determining the order of some prime numbers in \(\Z_r^\ast\).  In order to see this, suppose that \(e= p_1^{v_1}\ldots p_m^{v_m} \cdot r \) with \(\gcd(q-1,r)=1\) and \(v_1,\ldots, v_m \geq 0\). Then by  the Chinese Remainder Theorem the ring \(\Z_e\) is isomorphic to  \(\Z_{p_1^{v_1}}\times \ldots \times \Z_{p_m^{v_m}} \times  \Z_r\). To determine \(\vert \mathcal A _{\mod e}\vert\), in particular, we need to calculate the size of the multiplicative subgroup \(\langle p_1, \ldots, p_m\rangle  \) in \(\Z_r^\ast\).\\

The behaviour of Construction \ref{Construction: k-adic valuation of order}  allows us to discuss the selection of the integers \(k=p_1^{i_1}\cdots p_m^{i_m}\), \(i_1\geq 0, \ldots, i_m\geq 0\), for Construction \ref{New Construction} so that the number of multiple constructions of the same polynomial is reduced. First, we can obtain a naive upper bound on the exponents \(i_1,\ldots, i_m\) by computing Construction \ref{Construction: k-adic valuation of order} separately for every prime integer \(p_j\), \(1 \leq j \leq m\). Suppose then that the tail has a length of \(v_j\) and the orbit a length of \(s_j\), which is a divisor  of the multiplicative order \(\ord_{\nicefrac e {p_j^{v_j}}}\of{p_j}\). We set \(i_j \leq v_j + s_j\). We would like to note that if the order \(e\) of the initial polynomial \(f=m_\beta\) is known, the values \(v_j\) and \(s_j\) can be determined directly with Corollary \ref{Corollary: Order of f after Construction 2}. \\

 In order to eliminate the remaining duplicates, we suggest the following procedure: We select an integer \(k = p_1^{i_1}\cdots p_{m-1}^{i_{m-1}}\) with \(i_j\leq v_j+s_j\) for every \(1 \leq j \leq m-1\) and compute \(m_{\beta^k}\). 
 Then we construct the polynomials \(m_{\beta^{k\cdot p_m^i}}\) by applying Construction \ref{New Construction} for \(p_m\) repeatedly. With this we obtain  a tail \((m_{\beta^{k}}, \ldots, m_{\beta^{k\cdot p_m^{v_m-1}}})\) and an orbit \((m_{\beta^{k\cdot p_m^{v_m}}}, \ldots, m_{\beta^{k\cdot p_m^{v_m+(s-1)}}})\). Note that the length \(s\) of the orbit depends on \(k\).  \\
 
 Two integers \(k_1\) and \(k_2\) have either the same or a distinct tail. This will happen if and only if \(k_1 \equiv k_2\cdot q^j \mod e\) for an integer \(0 \leq j \leq n-1\).  Clearly if the tail is the same, the orbits coincide too. Thus, if the first tail polynomial is equal, the computation can be stopped.  The polynomials of the orbits of two different integers are also either distinct or equal. Equal orbits can also occur for integers with distinct tails. In this case the orbit polynomials appear in a shifted order. It is easy to see that any other integer of the form \( k_1 \cdot \left( \frac {k_2}{k_1}\right)^l\) with \(l \geq 0\) will yield the same orbit. For such integers we compute only the tail. 
 
 \begin{example}
 	As we have seen before, the number of constructed polynomials only depends on the order of the initial polynomial.  As an example for our computations we consider the polynomials \begin{align*}
 		f_1 &= X^8+X^5+X^3+X^2+a,\\
 		f_2 &= X^9+(a^2 +  a)X^8+(a ^3 + a^2)X^7+ aX^6+X^5+ (a^3 + a^2 + a)X^4\\
 		& \quad +( a^2 + a + 1)X^3+ a^2X^2+ a^3X+ a^3 + a^2 + a
 	\end{align*}  over 
 	\(\F_{16} = \F(a)\), where \(a\)  is a root of the monic irreducible polynomial \(X^4+X+1\) over \(\F_2\). \\ 
 	
 	The polynomial \(f_1\) is primitive and has order \(4.294.967.295=3 \cdot 5 \cdot 17 \cdot 257 \cdot 65537\). Construction \ref{New Construction} with  \(f_1\) as initial polynomial yields \(1.114.113\) monic irreducible polynomials of degree \(8\). Computing \(m_{\beta^k}\) for values of \(k\) of the form \(3^j\), \(j\geq 0\), and then applying the construction repeatedly for \(5\), there are \(33\) orbits of 32.768 polynomials each. The orbit for \(k=1\) contains 32.768 of the 67.108.864 monic irreducible polynomials of order \(3 \cdot 17 \cdot 257 \cdot 65537\) over \(\F_{16}\) and the other 32 orbits for \(k=3^j\), \(1\leq j \leq 32\), yield 1.048.576 of the 33.554.432 monic irreducible polynomials of order \(17 \cdot 257 \cdot 65537\) over \(\F_{16}\).  \(f_1\) has 5 non-zero coefficients and yields a weight distribution  of \(4^65^{384}6^{7225}7^{65997}8^{331084}9^{709417}\), which means that there exist 6 polynomials with smaller weight and 384 polynomials with the same weight. Hence, from these we could try to choose polynomials with other required properties that our initial polynomial might lack.\\
 	
 	   The polynomial \(f_2\) has order  \(68.719.476.735=3^3 \cdot 5 \cdot 7 \cdot 13 \cdot 19 \cdot 37 \cdot 73 \cdot109\) and yields \(4644\) monic irreducible polynomials over \(\F_{16}\) of degree \(9\) with the weight distribution \(6^{2}7^{47}8^{373}9^{1401}10^{2821}\). Even though the number of constructed polynomials is not very large, we could find polynomials of weight \(6\), \(7\) and \(8\). Considering the orbits for repeated application of Construction \ref{New Construction} for \(5\) with starting polynomials \(m_{\beta^k}\) with \(k=3^j\), \(j\geq 0\), there are 21 orbits of 216 polynomials each and the construction yields 3.888 of the 40.310.784 polynomials of order \( 509.033.161 = 7 \cdot 13 \cdot 19 \cdot 37 \cdot 73 \cdot 109\).\\

An interesting class of polynomials are the so-called \textit{normal polynomials} or \textit{\(N\)-polynomials} (see \cite{ Gathen1990, Gao1993, Kyuregyan2004, Meyn1995, Semaev1989}). A monic irreducible polynomial of degree \(n\) with a root \(\alpha\) is called \textit{normal}  if its roots \(\alpha, \alpha^q \ldots, \alpha^{q^{n-1}}\) are linearly independent over \(\Fq\) or, equivalently, if the degree of the greatest common divisor of the polynomials \(g_\alpha = \alpha X^{n-1}+ \alpha^q X^{n-2}+ \ldots+ \alpha^{q^{n-2}} X + \alpha^{q^{n-1}}\) and \(X^n-1\) over \(\Fqn\) is \(0\). This concept has been extended in \cite{Huczynska2013} to \textit{\(k\)-normal} polynomials which satisfy that the greatest common divisor of the two polynomials \(g_\alpha\) and \(X^n-1\) has degree \(k\).    Tables \ref{tab1} and \ref{tab2} show that Construction \ref{New Construction} also yields a large number of \(k\)-normal polynomials for small values of \(k\) which could be used for  respective applications. Since the number of \(k\)-polynomials decreases with \(k\) increasing, this distribution of \(k\)-normality is to be expected (see \cite{Huczynska2013}).
\begin{table}[h]
	\begin{center}
		\begin{minipage}{290pt} 
			\caption{Weight and \(k\)-normality distribution for \(f_1\)}\label{tab1}%
			\begin{tabular}{@{}llllll@{}}
				\toprule
				Weight & Total  & 0-normal & 1-normal & 2-normal & 3-normal \\
				\midrule
				4    & 6   & 1 & 5 & 0 & 0   \\
				5    & 384   & 139   & 240 & 5 & 0  \\
				6    & 7225   & 4160 & 2927 & 136 & 2 \\
				7 & 65997 & 47088 & 17746 & 1119  & 44\\
				8 & 331084 & 283554 & 44713 & 2625  & 192 \\
				9 & 709417 &  709417  & 0 & 0 & 0 \\
				\bottomrule
			\end{tabular}
		\end{minipage}
	\end{center}
\end{table}

 \begin{table}[h]
	\begin{center}
		\begin{minipage}{290pt}
			\caption{Weight and \(k\)-normality distribution for \(f_2\)}\label{tab2}%
			\begin{tabular}{@{}lllllll@{}}
				\toprule
				Weight & Total  & 0-normal & 1-normal & 2-normal & 3-normal & 4-normal\\
				\midrule
				6    & 2   & 1 &1  & 0 & 0 & 0 \\
				7 & 47 & 28 & 15 & 4 & 0 & 0 \\
				8 & 373 & 256& 102 & 14 & 1 & 0  \\
				9 & 1401 &  1091 & 290 & 18 & 0 & 2 \\
				10 & 2821 & 2475 & 339 & 5 & 2 & 0\\
				\bottomrule
			\end{tabular}
		\end{minipage}
	\end{center}
\end{table}

 \end{example}


\newpage






\begin{thebibliography}{}
	
	
	\bibitem{Albert1956}
	Albert, A.A.: Fundamental Concepts of Higher Algebra. University of Chicago Press, Chicago (1956)
	
	
	
	\bibitem{Cohen1969}
	Cohen, S.: On irreducible polynomials of certain types in finite fields. Mathematical Proceedings of the Cambridge Philosophical Society  \textbf{66}(2), 335--344 (1969)
	
	
	
	\bibitem{Cohen1992}
	Cohen, S.: The explicit construction of irreducible polynomials over finite fields. Designs, Codes and Cryptography \textbf{2}(2), 169--174 (1992) 
	
	\bibitem{Daykin1965}
	Daykin, D.E.: Generation of irreducible polynomials over a finite field. The American Mathematical Monthly \textbf{72}(6), 646--648 (1965)
	
	\bibitem{Gathen1990}
	Gathen, J.v.z., Giesbrecht, M.: Constructing normal bases in finite fields. Journal of Symbolic Computation \textbf{10} (6), 547--570 (1990)
	
	\bibitem{Gao1993}
	Gao, S.: Normal bases over finite fields. PhD Thesis (1993)
	
	\bibitem{Huczynska2013}
	Huczynska, S.,Mullen, G., Panario, D., Thomson, D.: Existence and properties of k-normal elements over finite fields. Finite Fields and Their Applications \textbf{24}, 170--183 (2013)
	
	
	
	
	\bibitem{Kyuregyan2006}
	Kyuregyan, M.: Recurrent methods for constructing irreducible polynomials over \(\F_q\) of odd characteristics. Finite Fields and Their Applications \textbf{12}(3), 357--378 (2006)
	
	
	
	\bibitem{Kyureghyan2010}
	Kyureghyan, G., Kyuregyan, M.: Irreducible compositions of polynomials over finite fields. Designs, Codes and Cryptography \textbf{61} (3), 301--314 (2011)
	
	
	
	
	\bibitem{Kyureghyan2020}
	Kyureghyan, G., Kyuregyan, M.: A recurrent construction of irreducible polynomials of fixed degree over finite fields. Applicable Algebra in Engineering, Communication and Computing (2020)
	
	
	\bibitem{Kyuregyan2004}
	Kyuregyan, M.: Iterated constructions of irreducible polynomials over finite fields with linearly independent roots. Finite fields and Their Applications \textbf{10}(3), 323--341 (2004)
	
	
	\bibitem{McNay1995}
	McNay, G.: Topics in finite fields. Ph.D. Thesis at the University of Glasgow (1995)
	
	\bibitem{Meyn1995}
	Meyn, H.: Explicit N-polynomials of 2-power degree over finite fields. Designs, Codes and Cryptography \textbf{6}(2), 107--116 (1995)
	
	
	\bibitem{Panario2020}
	Panario, D., Reis, L., Wang, Q.: Construction of irreducible polynomials through rational transformations. Journal of Pure and Applied Algebra \textbf{224}(5), 106241 (2020)
	
	
	
	\bibitem{Semaev1989}
	Semaev, I.: Construction of polynomials irreducible over a finite field with linearly independent roots. Mathematics of the USSR-Sbornik \textbf{63} (2), 507 (1989)
	
	\bibitem{Ugolini2013}
	Ugolini, S.: Sequences of binary irreducible polynomials. Discrete Mathematics \textbf{313} (22), 2656--2662 (2013).
	
	
	
	
	
	
	
\end{thebibliography}
\end{document}